\documentclass[sn-apa,iicol]{sn-jnl}

\usepackage{amsmath}
\usepackage{amsthm}
\usepackage{amssymb}
\usepackage{colortbl}
\usepackage{enumitem}
\usepackage{graphicx}
\usepackage{longtable}
\usepackage{mathtools}
\usepackage{supertabular}

\jyear{2022}%

\theoremstyle{thmstyleone}%
\newtheorem{theorem}{Theorem}
\newtheorem{corollary}[theorem]{Corollary}%

\theoremstyle{thmstyletwo}%

\theoremstyle{thmstylethree}%
\begin{document}

\title[Hybrid local search for CECSP]{A hybrid local search algorithm for the Continuous Energy-Constrained Scheduling Problem}

\author*[1]{\fnm{Roel} \sur{Brouwer}}\email{r.j.j.brouwer@uu.nl}

\author[1]{\fnm{Marjan} \spfx{van den} \sur{Akker}}\email{j.m.vandenakker@uu.nl}

\author[1]{\fnm{Han} \sur{Hoogeveen}}\email{j.a.hoogeveen@uu.nl}

\affil*[1]{\orgdiv{Department of Information and Computing Sciences}, \orgname{Utrecht University}, \orgaddress{\street{Princetonplein 5}, \postcode{3584 CC} \city{Utrecht}, \country{The Netherlands}}}

\abstract{We consider the Continuous Energy-Constrained Scheduling Problem (CECSP).
A set of jobs has to be processed on a continuous, shared resource. A schedule for a job consists of a start time, completion time, and a resource consumption profile. We want to find a schedule such that: each job does not start before its release time, is completed before its deadline, satisfies its full resource requirement, and respects its lower and upper bounds on resource consumption during processing. Our objective is to minimize the total weighted completion time.
We present a hybrid local search approach, using simulated annealing and linear programming, and compare it to a mixed-integer linear programming (MILP) formulation. We show that the hybrid local search approach matches the MILP formulation in solution quality for small instances, and is able to find a feasible solution for larger instances in reasonable time.}


\maketitle

\section{Introduction}\label{sec:intro}
When considering scheduling problems with resource constraints, jobs are often assumed to have a fixed duration or are not allowed to change the amount of resource they consume over time. In certain applications, however, these assumptions are too limiting. Typically, such cases can be found in areas where a (cumulative) amount of work (or resource) is required to complete a job, but the rate of consumption is not necessarily fixed, such as energy related applications. Think of demand-response in electricity consumption and electric vehicle charging, for example. To find schedules that allow for this, we consider an extension of the traditional scheduling problem, previously introduced as the Continuous Energy-Constrained Scheduling Problem (CECSP) by \cite{Nattaf2014}.

This problem is described as follows. A set $\{J_1,\ldots,J_n\}$ of jobs has to be processed on a continuous resource $P$. This means that, at any time, multiple jobs can be processed simultaneously and at different rates, as long as their total consumption does not exceed $P$. Each job $J_j$ requires a total amount of resource equal to $E_j$. We want to find a schedule such that: each job $J_j$ does not start before its release time $r_j$, is completed before its deadline $\bar{d}_j$, and respects its lower and upper bounds ($P^-_j,P^+_j$) on the resource consumption rate during processing. Preemption is not allowed: from its start until its completion, each job must consume resources at a rate of $P^-_j$ units. Our objective is to minimize the total weighted completion time.
We look at the case where both the resource and time are continuous. We assume that there is no efficiency function influencing resource consumption, and no explicit precedence relations exist between jobs. It should be noted that the presented approach can easily be extended to deal with (piece-wise) linear efficiency functions and precedence relations. The present work, however, does not include these extensions.

\vspace*{1em}
\textbf{Our contribution}. We present a hybrid local search approach (Section \ref{sec:hybrid-ls}), using simulated annealing and linear programming, to solve instances of the CECSP. In addition, we provide a mixed-integer linear programming (MILP) formulation (Section \ref{sec:milp}) and compare the performance of our hybrid approach to it in terms of quality and runtime. Finally, we found that the feasibility problem is solvable in polynomial time if we drop the lower bounds on resource consumption from the problem. We formulate a flow-based solution algorithm for this case (Section \ref{sec:flow}), that we use to select problem instances for the computational tests.

As far as we are aware, we are the first to propose a decomposition and try a combination of local search and mathematical programming techniques for CECSP. We show that it can compete with exact approaches on small instances, and has the potential to find good solutions for larger ones. Furthermore, the approach can be adapted to be applied to related (energy-constrained) scheduling problems.

\vspace*{1em}
The rest of this work is structured as follows. In Section \ref{sec:rel-work}, we discuss related work, followed by a detailed problem description in Section \ref{sec:prob-desc}. We explain our event-based modeling approach in Section \ref{sec:modeling}. From there, we build a MILP formulation in Section \ref{sec:milp}, and introduce our hybrid local search approach in Section \ref{sec:hybrid-ls}. We discuss how our test instances are generated in Section \ref{sec:instances}, which includes a description of our flow-based algorithm for the feasibility problem without lower bounds. In Section \ref{sec:results} we discuss the results, followed by the conclusions and suggestions for future work in Section \ref{sec:conclusion}.

\section{Related Work}\label{sec:rel-work}
The CECSP can be seen as a variant of the Resource Constrained Project Scheduling Problem (RCPSP). This problem concerns the scheduling of activities subject to precedence and resource constraints. The surveys by \cite{Hartmann2010,Hartmann2022} provide a good overview.
\cite{Baptiste1999} defined the Cumulative Scheduling Problem (CuSP) as a subproblem of the RCPSP. In the case of CuSP, we have a single resource with a given capacity, and we need to schedule a number of activities without exceeding this capacity. Each activity has a release time, deadline, (fixed) processing time and (constant) resource capacity requirement. More recent work further defined the CECSP as a generalization of CuSP \citep{Nattaf2014} and provided a hybrid branch-and-bound algorithm \citep{Nattaf2017} and mixed-integer linear program \citep{Nattaf2019} to find exact solutions for small instances. In the CECSP, the resource capacity requirement is a range with a lower and upper bound, rather than a fixed value, and the consumption rate can vary during the execution of the job. As a result, the processing time is no longer fixed, but depends on the consumption rate during its execution. Through its relation to CuSP, \cite{Nattaf2016} proved that finding a feasible solution for CECSP is NP-complete.

The CECSP is also closely related to the scheduling of malleable jobs (as introduced by \cite{Turek1992}) on parallel machines, which involves the scheduling of jobs on $P$ machines, while the number of machines assigned to a job can change during its execution.

\begin{table}
	\begin{center}
		\caption{Example instance of the CECSP with three jobs}
		\label{tab:example-instance}
		\begin{tabular}{@{}l|lllllll@{}}
			\toprule
			$j$ & $E_j$ & $r_j$ & $\bar{d}_j$ & $P^-_j$ & $P^+_j$ & $w_j$ & $B_j$ \\
			\midrule
			1 & 70.0 & 0.0 & 3.0 & 10.0 & 30.0 & 1.0 & 0.0 \\
			2 & 20.0 & 1.5 & 3.0 & 10.0 & 40.0 & 3.5 & 0.0 \\
			3 & 45.0 & 2.5 & 4.0 & 10.0 & 50.0 & 5.0 & 0.0 \\
			\botrule
		\end{tabular}
	\end{center}
\end{table}

\begin{figure}
	\centering
	\includegraphics[width=4.5cm]{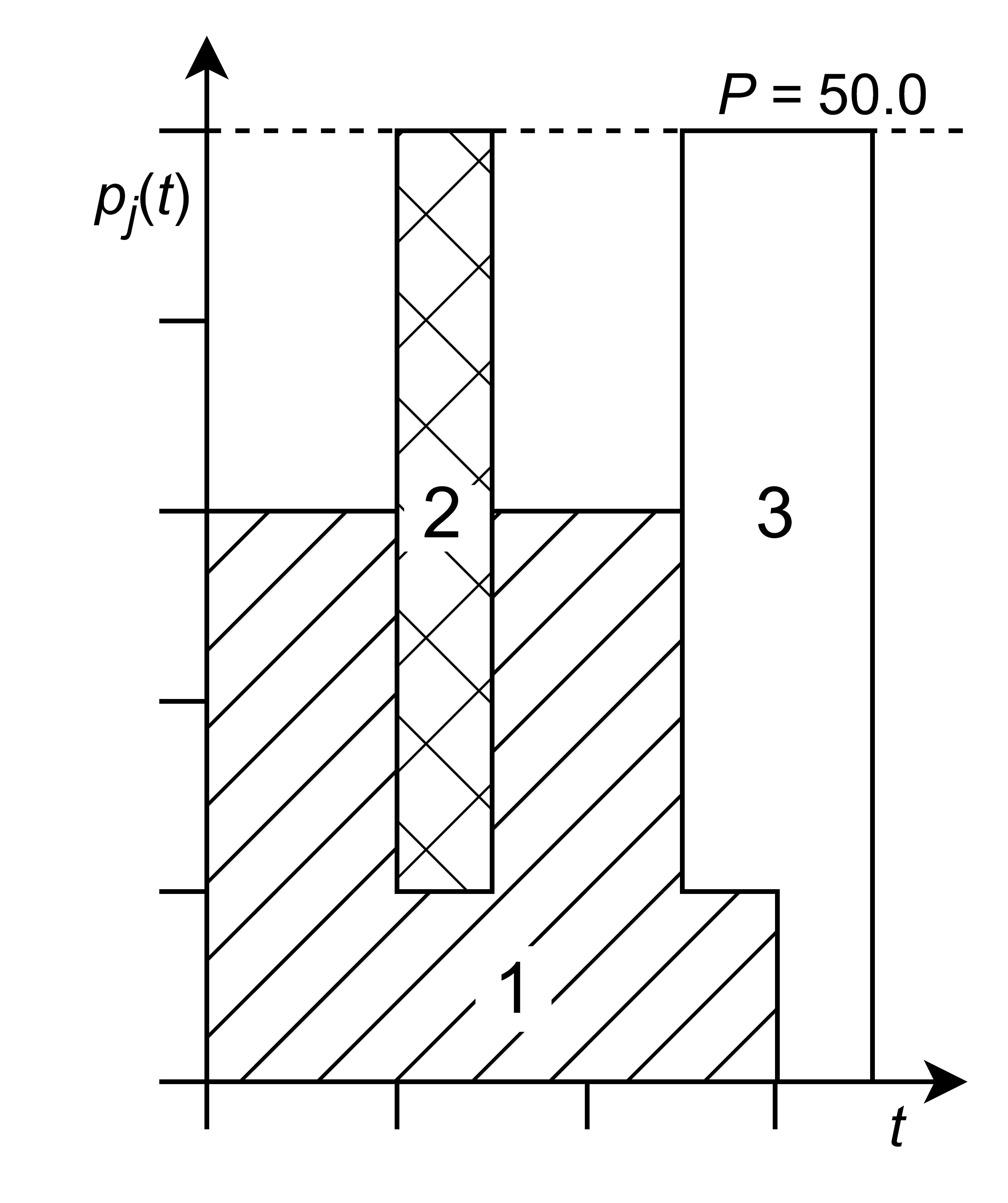}
	\caption{Optimal schedule for the example instance with three jobs}
	\label{fig:example-instance}
\end{figure}

\section{Problem description}\label{sec:prob-desc}
An instance of the CECSP problem is defined by the following properties:
\begin{itemize}
	\item Available resource: constant $P$;
	\item For each job $J_j,$ $j \in \{1, ..., n \}$:
	\begin{itemize}
		\item Resource requirement $E_j$;
		\item Release time $r_j$;
		\item Deadline $\bar{d}_j$;
		\item Lower bound $P^-_j$;
		\item Upper bound $P^+_j$;
		\item Linear cost function $w_j \cdot C_j + B_j$.
	\end{itemize}
\end{itemize}
For each job, we need to determine a start time $S_j$, a completion time $C_j$ and a resource consumption profile $p_j(t)$ such that $\sum_j w_j C_j + B_j$ is minimal, while the following constraints are respected:
\begin{itemize}
	\item[C1] The amount of resource a job $j$ consumes is exactly $E_j$;
	\item[C2] A job $j$ does not start before its release time $r_j$, i.e. $S_j \geq r_j$;
	\item[C3] A job $j$ completes no later than its deadline $\bar{d}_j$, i.e. $C_j \leq \bar{d}_j$;
	\item[C4] A job $j$ only consumes resources between its start and completion time, i.e. $p_j(t) = 0$ for $t < S_j$ or $t \geq C_j$;
	\item[C5] While active, the resource consumption of a job $j$ never drops below $P_j^-$ or rises above $P_j^+$, i.e. $P_j^- \leq p_j(t) \leq P_j^+$ for $S_j \leq t \leq C_j$;
	\item[C6] The total amount of resource consumed by all jobs together at any given time can never exceed $P$, i.e. $\sum_j p_j(t) \leq P$ for all possible values of $t$.
\end{itemize}

An example instance with three jobs and $P = 50.00$ is provided in Table \ref{tab:example-instance}. Figure \ref{fig:example-instance} gives a visual representation of the optimal schedule for this instance.

Comparing the problem description above to the version originally introduced by \cite{Nattaf2014}, we notice two differences:
\begin{enumerate}
	\item Our version does not consider efficiency functions (i.e., the amount of work done is equal to the amount of resource consumed), while in the work of Nattaf et al., (piece-wise) linear efficiency functions are considered (i.e., the amount of work done is a linear function of the amount of resource consumed).
	\item Our objective function minimizes the total weighted completion time, while the objective in the work of Nattaf et al. is to minimize the amount of resource consumed.
\end{enumerate}

\section{Modeling: event-based approach}\label{sec:modeling}
In this work, we will use an event-based model of the problem. This means that we define a schedule by determining the timing of \textit{events}, and the activity of jobs during the \textit{intervals} that are bounded by these events.

We consider two types of events, associated with a job $J_j$: its start $S_j$ and completion $C_j$. Now let us consider a list $\mathcal{E}$, that lists all events in the order that they occur over time.

For a given order of events, we can divide the schedule into $2n-1$ intervals, where each interval is bounded by two consecutive events. In the following, when we refer to an interval, we always mean the span of time in between two consecutive events. So, no event ever occurs during an interval. Also note that an interval can be empty, if it is bounded by two events that happen at the same time.

During these intervals, the set of jobs that are being processed does not change. We can show that there always exists an optimal solution in which $p_j(t)$ remains constant during each interval. We can therefore limit ourselves to such solutions.

\begin{theorem}
	For any feasible schedule $\mathcal{S}$ (with start times $S_j$, completion times $C_j$ and resource consumption profiles $p_j(t)$ for all jobs $J_j, j \in \{1, ..., n\}$) that follows a given event order $\mathcal{E}$, a feasible schedule $\mathcal{S}'$ exists with the same objective value where the resource consumption $p_j'(t)$ of all jobs remains constant during each interval.
\end{theorem}
\begin{proof}
	Consider any interval between two events, $i$ and $i'$. The interval starts at $t_i$ and ends at $t_{i'}$.
	
	A job $J_j$ is inactive during this interval if $t_i \geq C_j$ or $t_{i'} \leq S_j$ in the schedule, and active otherwise.
	
	The resource consumption profile of any active job $J_j$ in $\mathcal{S}$ within this interval does not violate the lower and upper bounds for this job: $P^-_j \leq p_j(t) \leq P^+_j \ \forall t, t_i \leq t \leq t_{i'}$.
	
	Therefore, the average consumption of job $J_j$ within this interval, $\bar{p}_{j}= \frac{\int_{t_i}^{t_{i'}} p_j(t)\ dt}{t_{i'} - t_i}$, does not exceed the bounds on consumption for job $J_j$: $P^-_j \leq \bar{p}_{j} \leq P^+_j$.
	
	In this interval, all active jobs in $\mathcal{S}$ together consume at most ($t_{i'}-t_i)P$ of the resource, i.e.:
	
	\[\sum_j\bar{p}_{j} (t_{i'}-t_i) \leq (t_{i'}-t_i)P \]
	
	From this, it follows that $\sum_j\bar{p}_{j} \leq P$, i.e., the sum of the average consumption of all jobs does not exceed the available amount of resource.
	
	We obtain $\mathcal{S}'$, from $\mathcal{S}$ by replacing all resource consumption profiles $p_j(t)$ with a step-wise constant function $p_j'(t)$ that remains constant during any interval (in between two consecutive events), and takes the value $\bar{p}_{j}$, i.e. the average resource consumption of job $J_j$ in $\mathcal{S}$, during that interval.
	
	We have shown that the schedule $\mathcal{S}'$ is indeed feasible. For each job, the resulting consumption rate is within the bounds allowed for that job and the total amount of available resource is not exceeded either. No other constraints are affected.
	
	It remains to show that $\mathcal{S}'$ has the same objective value. The objective depends only on the value of $C_j$. As the timing of all events in $\mathcal{S}'$ is identical to that in $\mathcal{S}$, the objective is not affected by the modifications to the schedule.
\end{proof}
\begin{corollary}
	For any optimal schedule $\mathcal{S}^*$, there exists an optimal schedule ${\mathcal{S}^*}'$ where the resource consumption ${p_j^*}'(t)$ of all jobs remains constant during each interval.
\end{corollary}

This means that we can ignore solutions with a non-constant resource consumption profile for any job during any interval. Therefore, we can fully represent a resource consumption profile of job $J_j$ with at most $2n-1$ values $p_{j,i}$, one for each interval in which $J_j$ is being processed, indicating the total amount of resource consumed during that interval.

In the following, we will maintain a link between events and jobs by their index. Every job has two associated events. For a job with index $j$, we will say that its start event has index $i = 2j - 1$, and its completion event has index $i + 1 = 2j$. This allows us to uniquely identify the start and completion events of any job easily in notation. \emph{The index of an event $i$ does not indicate its position in the event order.} We will maintain this indexing scheme throughout this work.

Furthermore, intervals are defined by the event at the start of that interval. For example, let $i$ and $i'$ be the indices of the first two events. The first interval starts when the first event happens: at $t_i$; and ends when the next event starts: at $t_{i'}$. Then, $p_{j,i}$ represents the amount of resource that job $J_j$ consumes in the interval associated with event $i$, spanning the time between $t_i$ and $t_{i'}$.

\section{Mixed-integer linear programming formulation}\label{sec:milp}
Using the event-based approach, we developed a MILP formulation that will find an optimal schedule. In this section, we gradually build this formulation.

Although we arrived at it independently, our formulation is similar to the event based formulations presented by \cite{Nattaf2016}. The models share their focus on intervals, delimited by events, but differ in how the order of these events is expressed. \cite{Nattaf2016} use assignment variables for this: binary variables that indicate whether events occur at a given position in the order, or whether a job is active during a given interval. Our model uses binary ordering variables, to indicate the relative order between pairs of events.

First, we define the decision variables of our model. For each of the $2n$ events $i$, $t_i \geq 0$ will denote the time at which it takes place. 
For each job $J_j$ and event $i$, $p_{j,i} \geq 0$ is the amount of resource that $J_j$ consumes during the interval defined by event $i$. We know the consumption of a job $J_j$ during the interval it completes to be 0 and therefore set $p_{j, 2j} = 0$.
Note that we do not know beforehand what event will be the last in the event order. The interval that is defined by that event is, theoretically, open-ended. However, in the formulations below you will note that any $p_{j,i}$ associated with this interval is forced to 0, as all jobs complete in or before the interval.

Additionally, we define binary variables $a_{i,i'}$ and $b_{i,i'}$ to determine the order of events. These can be viewed as helper variables.
For two events $i$ and $i'$, $a_{i,i'}$ represents the order they occur in. If $a_{i,i'} = 1$, then $i$ occurs before $i'$ in the event order. 
On top of that, $b_{i,i'}$ indicates whether an event is an immediate successor of another: $b_{i,i'} = 1$ if $i'$ occurs immediately after $i$.
In our formulation, we will use $a_{i,i'}$ where possible and $b_{i,i'}$ if we have to: exclusively for the enforcement of lower bounds.

The objective is to minimize the combined cost functions of the jobs:

\begin{equation}
	\min \sum\limits_{j=1}^n\left(w_j t_{2j} + B_j\right)
	\label{eq:lp-objective-mip}
\end{equation}

Now, let us go through the constraints we listed in Section \ref{sec:prob-desc} above. 

First, we ensure that the total amount of resource consumed by each job $J_j$ equals $E_j$ (C1), $\forall j \in \{1, ..., n\}$:

\begin{equation}
	\sum\limits_{i = 1}^{2n} p_{j,i} = E_j \label{eq:lp-totalwork-mip}
\end{equation}

Next, we enforce the release time (C2) and deadline (C3) for each job $J_j$. Since we can identify the start and completion event of a job by its index $i$, we obtain that $\forall j \in \{1, ..., n\}$:

\begin{equation}
	t_{2j - 1} \geq r_j \label{eq:lp-releasedate-mip}
\end{equation}
\begin{equation}
	t_{2j} \leq \bar{d}_j \label{eq:lp-deadline-mip}
\end{equation}

We ensure that $p_{j,i} = 0$ for any interval after the completion event of $J_j$ (with index $2j$) has happened or before its start event (with index $2j - 1$) has taken place (C4). We use a big $M$ constraint with the order variables $a_{i,i'}$ to enforce this exactly for those intervals. So $\forall j \in \{1, ..., n\}, i \in \{1, ..., 2n\}$, and sufficiently large $M$ ($M=E_j$ will suffice) we have:

\begin{equation}
	p_{j,i} \leq a_{i, 2j}M \label{eq:lp-powerupperbound2-mip}
\end{equation}
\begin{equation}
	p_{j,i} \leq a_{2j - 1, i}M \label{eq:lp-powerupperbound3-mip}
\end{equation}

In between the start and completion events of a job $J_j$, we need to choose values for $p_{j,i}$ that respect the lower ($P^-_j$) and upper ($P^+_j$) bounds on resource consumption (C5). The upper bound can be enforced in between any two events. We use a big $M$ constraint with order variable $a_{i,i'}$ to activate the constraint only for events that happen in the right chronological order, hence $\forall j \in \{1, ..., n\}, i, i' \in \{1, ..., 2n\}, i \neq i'$, and sufficiently large $M$:

\begin{equation}
	p_{j,i} \leq P_j^+ (t_{i'} - t_{i}) + Ma_{i',i} \label{eq:lp-powerupperbound-mip}
\end{equation}

 The lower bound only applies for consecutive events, which both fall within the execution window of job $J_j$. We use a big $M$ term with the $b_{i,i'}$ helper variable to activate the constraint only for events that are immediate successors. Besides that, we include the variables $a_{i',2j-1}$ and $a_{2j,i}$ to make sure the constraint is only active if $i$ and $i'$ are between the start event of job $J_j$ (index $2j - 1$) and its completion (index $2j$). Hence $\forall j \in \{1, ..., n\}, i, i' \in \{1, ..., 2n\}, i \neq i'$, and sufficiently large $M$:

\begin{equation}
\begin{aligned}
	p_{j,i} \geq & & P_j^- (t_{i'} - t_{i}) \\ & & - (1-b_{i,i'} + a_{i',2j-1} + a_{2j,i})M \label{eq:lp-powerlowerbound-mip}
\end{aligned}
\end{equation}

During an interval the total resource consumption of all busy jobs cannot exceed $P$ at any time. This has to be enforced for all intervals, $\forall i, i' \in \{1, ..., 2n\}, i \neq i'$, and sufficiently large $M$:

\begin{equation}
	\sum\limits_{j=1}^n p_{j,i} \leq P(t_{i'} - t_{i}) + Ma_{i',i} \label{eq:lp-powercapacity-mip}
\end{equation}

We need to make sure that the event times are consistent with the variables $a_{i,i'}$ that control the order of events.
Therefore, $\forall i, i' \in \{1, ..., 2n\}, i \neq i'$, and sufficiently large $M$, we get:

\begin{equation}
	t_i \leq t_{i'} + Ma_{i',i} \label{eq:lp-intervalorder-mip}
\end{equation}

We ensure that either $a_{i,i'} = 1$ or $a_{i',i} = 1$, but not both, for any pair of distinct events. We could, of course, get rid of half of our binary variables by expressing $a_{i',i}$ in terms of $a_{i,i'}$ for all $i' > i$. We choose not to do so for ease of notation, therefore we need to ensure that $\forall i, i' \in \{1, ..., 2n\}, i < i'$:

\begin{equation}
	a_{i,i'} + a_{i',i} = 1 \label{eq:lp-order-doubles-mip}
\end{equation}

We can note here that we also already know: $\forall j\in \{1, ..., n\}$, $a_{2j - 1,2j} = 1$.

Finally, we need some constraints to make sure the $b_{i,i'}$ variables indicate events that are immediate successors. First making sure that $b_{i,i'}$ can only be 1 if $i'$ occurs immediately after $i$, $\forall i, i' \in \{1, ..., 2n\}, i \neq i'$, and sufficiently large $M$:

\begin{equation}
	\sum\limits_{i'' = 1}^{2n} a_{i,i''} - \sum\limits_{i'' = 1}^{2n} a_{i',i''} \leq 1 + (1 - b_{i,i'})M \label{eq:lp-order-succ1-mip}
\end{equation}
\begin{equation}
	\sum\limits_{i''=1}^{2n}a_{i,i''} - \sum\limits_{i''=1}^{2n} a_{i',i''} \geq 1 - (1 - b_{i,i'})M \label{eq:lp-order-succ2-mip}
\end{equation}

And second, by requiring that exactly (or equivalently: at least) $2n-1$ variables $b_{i,i'}$ actually take the value of 1:

\begin{equation}
	\sum_{i=1}^{2n} \sum_{i' \neq i} b_{i,i'}= 2n - 1 \label{eq:lp-order-succ3-mip} \\
\end{equation}

This completes the initial MILP formulation. However, we have added two types of valid inequalities to speed up the process of solving the model. We know that the range of feasible processing times for a job is limited by the lower and upper bound on its resource consumption, $\forall j \in \{1, ..., n\}$:

\begin{equation}
	t_{2j} - t_{2j - 1} \leq E_j / P^-_j
	\label{eq:lp-valid-upper-limit-mip}
\end{equation}
\begin{equation}
	t_{2j} - t_{2j - 1} \geq E_j / P+_j \label{eq:lp-valid-lower-limit-mip}
\end{equation}

Indeed, the addition of these inequalities has a significant impact on the time needed to find an optimal solution.
The MILP formulation that our hybrid local search approach will be tested against consists of all equations \eqref{eq:lp-objective-mip} - \eqref{eq:lp-valid-lower-limit-mip}.

\section{Hybrid local search}\label{sec:hybrid-ls}
In the MILP formulation described above, we can identify two separable parts to the problem. The binary variables $a_{i,i'}$ (and $b_{i,i'}$) fully determine the order of events $\mathcal{E}$. Then, the continuous variables $t_i$ and $p_{j,i}$ determine the exact schedule for this order.

We can decompose the problem into these two parts. Once we have an order of events $\mathcal{E}$, determining an optimal schedule boils down to solving an LP, as all binary variables disappear from the formulation once the event order is fixed.

A global overview of our approach is given in Algorithm \ref{alg:overview}. The elements in this overview will be discussed in more detail in the following sections.

\begin{algorithm}
	\caption{Hybrid local search approach}\label{alg:overview}
	\begin{algorithmic}[1]
		\State $prec \gets $ ImplicitPrecedences($J$)
		\State $\mathcal{E}, best\_score \gets $ InitialSolution($J$, $P$, $prec$)
		\State $T \gets T_{init}$
		\For{$iter \in \{1, ..., max\_iter\}$}
			\State $\mathcal{E}, score \gets $ GetNeighbor($\mathcal{E}$, $T$, $J$, $P$, $prec$)
			\If{$score < best\_score$}
				\State $\mathcal{E}_{best}, best\_score \gets \mathcal{E}, score$
			\EndIf
			\If{$iter \% \alpha_{period} = 0$}
				\State $T \gets T \cdot \alpha$
			\EndIf
		\EndFor
		\State \Return{$\mathcal{E}_{best}, best\_score$}
	\end{algorithmic}
\end{algorithm}

\subsection{Finding an optimal schedule for a given order: linear program}
We will adapt the MILP formulation described above to solve only the (sub)problem of finding an optimal schedule for a given event order.

The LP will consist of Equation \eqref{eq:lp-objective-mip} - \eqref{eq:lp-deadline-mip} and \eqref{eq:lp-powerupperbound-mip} - \eqref{eq:lp-intervalorder-mip}, eliminating any part of the inequalities that contain the variables $a_{i,i'}$ and $b_{i,i'}$.

Instead of enforcing $p_{j,i}$ to be 0 outside of $J_j$'s processing window (as in Equation \eqref{eq:lp-powerupperbound2-mip} and \eqref{eq:lp-powerupperbound3-mip}), we only take into account $p_{j,i}$ for $i$ within the processing window, which can easily be done as we already know the order of events. Similarly, upper and lower bounds are only enforced on intervals within the processing window.

For ease of notation, let us define two functions:
\begin{itemize}
	\item $I(e)$ gives the event index $i$ of the event in position $e$ of the event order;
	\item $E(i)$ gives the position in the event order $e$ of the event with index $i$.
\end{itemize}
Note that these functions are known already before we attempt to solve the LP, as they only depend on the event order.

So, for example the inequality enforcing the order of events (Equation \eqref{eq:lp-intervalorder-mip} in the MILP formulation) now becomes, $\forall e \in \{1, ..., 2n - 1\}$:
\begin{equation}
	t_{I(e)} \leq t_{I(e + 1)} \label{eq:lp-intervalorder}
\end{equation}
And the inequality controlling the total work of a job (Equation \eqref{eq:lp-totalwork-mip} in the MILP formulation) now becomes, $\forall j \in \{1, ..., n\}$:
\begin{equation}
	\sum\limits_{e \in \{E(2j - 1), ..., E(2j) - 1\}} p_{j,I(e)} = E_j \label{eq:lp-totalwork}
\end{equation}
Note that $\{E(2j - 1), ..., E(2j) - 1\}$ represents the set of positions in the event order between the start and completion event of a job $J_j$, including the former and excluding the latter.

Finally, to evaluate the relative quality of infeasible event orders, we allow for the violation of lower bounds, upper bounds, and/or resource availability constraints, with appropriate cost. For this purpose, we introduce three types of slack variables:
\begin{itemize}
	\item $s^-_{j,i}$ allowing the resource assignment to job $J_j$ in interval $i$ to go below its lower bound $P^-_j$;
	\item $s^+_{j,i}$ allowing the resource assignment to job $J_j$ in interval $i$ to go over its upper bound $P^+_j$;
	\item $s^t_i$ allowing for increasing the amount of available resource during an interval $i$.
\end{itemize}
This affects the inequalities enforcing these bounds in the LP (Equation \eqref{eq:lp-powerupperbound-mip} - \eqref{eq:lp-powercapacity-mip} in the MILP formulation), $\forall j \in \{1, ..., n\}, e \in \{E(2j - 1), ..., E(2j) - 1\}$:
\begin{equation}
	p_{j,I(e)} \geq P_j^- (t_{I(e+1)} - t_{I(e)}) -s^-_{j,I(e)} \label{eq:lp-powerlowerbound}
\end{equation}
\begin{equation}
	p_{j,I(e)} \leq P_j^+ (t_{I(e+1)} - t_{I(e)}) +s^+_{j,I(e)} \label{eq:lp-powerupperbound}
\end{equation}
And, $\forall e \in \{1, ..., 2n - 1\}$:
\begin{equation}
	\sum\limits_{j=1}^n p_{j,I(e)} \leq P(t_{I(e+1)} - t_{I(e)}) + s^t_{I(e)} \label{eq:lp-powercapacity}
\end{equation}

These slack variables are penalized in the objective function, using $L^B$ for violation of upper and lower bounds and $L^R$ for excessive resource usage. The slack variables are non-negative, but do not have an enforced upper bound. This adds the following term to the objective:
\begin{equation}
	\sum\limits_{i = 1}^{2n-1} \left(L^R s^t_i + \sum\limits_{j=1}^n L^B(s^-_{j,i} + s^+_{j,i})\right) \label{eq:lp-objective-add}
\end{equation}

This allows us to get an idea of the quality of infeasible event orders that we can use to guide the local search towards solutions that are closer to feasibility. For many event orders, no feasible schedule exists. Scoring the `degree of infeasibility', helps us to get from an infeasible order to a feasible one.


\subsection{Finding a good event order: local search}
It remains to find good event orders to feed to the LP. For this, we use a local search algorithm: simulated annealing. Below, we will describe the core of the local search, and the neighborhood operators that are used. 

We start our local search with the event order that results from the greedy algorithm (see below). Every iteration, we modify our current solution using one of the neighborhood operators. Improvements are always accepted, while solutions that are worse than the current one are accepted with a probability determined by the current value of the temperature $T$. The precise settings for the simulated annealing will be discussed in Section \ref{sec:results}.

If a candidate solution is rejected, we keep trying with the same neighborhood operator. We do this by generating a random permutation of the event list each time we start with a neighborhood operator, and apply the operator to the events in that order, until a candidate solution is accepted. If this is unsuccessful we move on to the next neighborhood operator. If no solution is accepted after looping through all three neighborhood operators, we terminate the search and report the best seen solution up to that point.

We have three neighborhood operators that we use to modify the event order and obtain new candidate solutions:
\begin{itemize}
	\item \textbf{Swap two adjacent events}. We take the next event from the random permutation, at position $e \in \{1, ...,2n-1\}$ and its successor at position $e + 1$ and reverse their order. This means that the event originally at position $e$ now is at position $e + 1$ and vice versa.
	\item \textbf{Move a single event}. We take the next event from the random permutation, at position $e \in \{1, ...,2n\}$. We then determine its range of movement. On both sides, it is limited by the first event we encounter with which the event has a precedence relation. How we determine these precedence relations is discussed in Section \ref{sec:init-sol}. From this range, we then select a new position, where the probability of a position being selected is inversely proportional to its distance to the original position of the event. That is, smaller displacements are more likely to be selected than larger ones. The event is then moved to its new position, with the other events in between the current and former position of the event shifted accordingly.
	\item \textbf{Move both events associated with a single job}. We take the next job from the random permutation, job $J_j$ $j \in \{1, ...,n\}$. We then look at both of its associated events, and determine their combined range of movement. As before, we determine the number of positions both can move in either direction without crossing over an event with which they have a precedence relation (see Section \ref{sec:init-sol}). We then take the minimum of these values to get a combined range. From this range, we select a new position with uniform probability. The events are then moved to their new position, both offset from their original position by the same amount, with the other events in between the current and former position of the events shifted accordingly.
\end{itemize}
\vspace*{1em}
Each of these neighborhood operators results in a small modification of the event order, necessitating an update of the LP for evaluation. We modify only those parts of the LP that are affected, and resolve the model to obtain the score of the new candidate solution, starting from the current solution.

\subsection{Preprocessing and initial solution: greedy algorithm}\label{sec:init-sol}
We obtain an initial event order using a greedy algorithm. The algorithm assigns resource to jobs that have been released and have not been fully served, and extracts the order of events from the resulting schedule. In the construction of the initial order, we ignore lower bounds on resource consumption and do not consider the cost function at all. Keep in mind that we are only interested in the order of events at this point, not in the schedule itself.

We take the release times $r_j$ and deadlines $\bar{d}_j$ of all jobs, and sort them in ascending order. In this way, we obtain a sequence of time periods during which the set of jobs available for processing remains the same.

We loop over these time periods, keeping track of the list of available jobs during the period. Then, we decide the assignment of resources during each time period in the following way:
\begin{enumerate}
	\item We assign all available jobs their minimal amount of resource required. If this exceeds the total amount of resource available, we stop after this step and continue to the next time period;
	\item In order of increasing deadline, we assign as much resource to each job as it can consume. This may be bounded by the upper bound on resource consumption of the job, or by the amount of resource that is left to be assigned at this point. We continue until the resource in the time period is fully used or no other jobs are available.
\end{enumerate}
The minimal amount of resource a job has to consume is given by $\max\{0, \tilde{E}_j - (\bar{d}_j - t_{end})P^+_j\}$, while its maximum is given by $\min\{\tilde{E}_j, P^+_j(t_{end}-t_{start})\}$. Here $t_{start}$, $t_{end}$ are the start and end of the current time period and $\tilde{E}_j$ is the \emph{remaining} resource requirement of the job.

At any time during this procedure:
\begin{itemize}
	\item If a job is assigned resources for the first time, we append the start event for that job to our current list;
	\item If a job's full resource requirement is satisfied, we append the completion event for that job to our current list.
\end{itemize}

At the end of the procedure we have a complete order of events, which will be the starting solution for our local search. We determine the quality of the optimal schedule for this initial event order by solving the LP above. This schedule is likely to assign resources differently from the greedy algorithm we used to find the event order in the first place. Note that the schedule may have positive penalty terms, i.e. it does not have to meet the available resource bound $P$ or the job lower or upper bounds $P_j^-$ and $P_j^+$.\\

To complete the description of our approach, we will discuss the preprocessing steps that we use in our algorithm. We deduce a number of implicit precedence constraints that we use in the neighborhood operators:
\begin{enumerate}
	\item A start event of a job $S_j$ always has to occur before its corresponding completion event $C_j$;
	\item Let $u_j = E_j / P^+_j$ be the minimal processing time required to complete a job. Then the possible time ranges of a start and completion event are $[r_j, \bar{d}_j - u_j]$ and $[r_j + u_j, \bar{d}_j]$, respectively. For any two events, if the last possible time within the range of event $i$ lies before the first possible time within the range of event $i'$, we know that $i$ has to come before $i'$ in any event order.
\end{enumerate}

\section{Test instances}\label{sec:instances}
We generated a number of instances to evaluate our approach. We consider three important characteristics to group these instances: the number of jobs ($n$), the amount of resource ($P$) and whether an instance is  \textit{adversarial} ($a$).

An instance is said to be adversarial, when the ordering of weights is such that the weight of jobs monotonically increases with increasing deadline. This is meant as a counter to the greedy algorithm that determines a starting solution, as this assigns resources to jobs in order of increasing deadline.

By varying $P$, we also affect the ratios of $P/P^-_j$, $P/P^+_j$ and $P/E_j$, as the sampling of values for $E_j$, $P^+_j$ and $P^-_j$ is (largely) unaffected by the chosen values of $n$, $P$ and $a$. Each of these ratios are more interesting properties of an instance than $P$ in itself.

Table \ref{tab:instance-overview} lists the values of $n$, $P$ and $a$ for which instances were generated. Each combination of $n$, $P$ and $a$ was used to generate four unique instances.

Below we describe the procedure for generating instances.

\begin{table}
	\begin{center}
		\caption{Overview of values of $n$, $P$ and $a$ used for generating instances}
		\label{tab:instance-overview}
	\begin{tabular}{@{}ll@{}}
		\toprule
		$n$ & 5, 10, 15, 20, 30, 50 \\
		$P$ & 25.00, 50.00, 100.00, 200.00 \\
		$a$ & True, False \\
		\botrule
	\end{tabular}
	\end{center}
\end{table}

\subsection{Instance generation method}
The instance generation is based on the three parameters discussed above ($n$, $P$ and $a$) as well as four scaling parameters: $a_{\mathrm{maxlow}}$, $a_{\mathrm{minupp}}$, $a_{\mathrm{rshift}}$ and $a_{\mathrm{pws}}$. For each job we sample seven values, that together fully determine the six properties of a job as described in Section \ref{sec:prob-desc}.

\begin{itemize}
	\item A resource requirement $E_j$ from $U(10, 100)$.
	\item A lower bound $P^-_j$ from $U(0, \min\{a_{\mathrm{maxlow}}P, a_{\mathrm{minupp}}E_j\})$. The parameters $a_{\mathrm{maxlow}}$ and $a_{\mathrm{minupp}}$ can be used to control the largest possible lower bound. Here $a_{\mathrm{maxlow}}$ makes sure that the lower bound on the resource consumption per time unit never exceeds a fraction ($a_{\mathrm{maxlow}}$) of what is available during a unit of time, leading to instances where the lower bounds do not restrict the schedule to only processing a single job at a time. Then, $a_{\mathrm{minupp}}$ has a similar purpose, but as a fraction of the total resource requirement of a job, ensuring that the lower bound does not restrict the possible duration of a job too much.
	\item An upper bound $P^+_j$ from $U(a_{\mathrm{minupp}}E_j, E_j)$. The $a_{\mathrm{minupp}}$ parameter can be used to make the upper bound more likely to be more restrictive, while also ensuring it does not drop below the lower bound. This distribution also ensures that no job can be completed in less than a single time unit.
	\item A release time $r_j$ from $U\left(-a_{\mathrm{rshift}} \frac{\sum_j E_j}{P}, (1-a_{\mathrm{rshift}})\frac{\sum_j E_j}{P}\right)$. Any negative values are treated as zero. The parameter $a_{\mathrm{rshift}}$ shifts the window for generating release times, such that a fraction of the generated release times equal to $a_{\mathrm{rshift}}$ is expected to be a negative number. In this way, one can control the expected number of jobs that are available for processing immediately at time 0. The largest possible release time is limited by the expected total processing time (at full utilization).
	\item a deadline $\bar{d}_j$: $r_j+ U\left(\frac{E_j}{\min\{P, P^+_j\}}, a_{\mathrm{pws}}\frac{\sum_j E_j}{P}\right)$. The parameter $a_{\mathrm{pws}}$ sets the latest possible deadline in terms of the expected total processing time (at full utilization). The earliest possible deadline is the release time plus the minimum required processing time (its requirement divided by the minimum of its upper bound or the amount of resource available per unit of time), avoiding trivially infeasible instances.
	\item A weight $w_j$ from $U(0, 5)$ and an optional objective constant $B_j$ from $U(0, 10)$. Together they determine the cost function of job $J_j$: $w_j \cdot C_j + B_j$.
\end{itemize}

All numbers are rounded to two decimal places.

This instance generation method 
does not entirely prevent that there might be a shortage of available resource during a certain period of time, depending on the sampled release times and deadlines. The odds of this happening can be reduced by increasing the range of release times and/or processing windows, using $a_{\mathrm{pws}}$ for example. We describe a method that gives a strong indication of the feasibility of in instance in the following section. In Section \ref{sec:param-instgen} we describe our choice of parameters that lead to mostly feasible instances.

\subsection{Checking feasibility}\label{sec:flow}
To get an indication of the (in)feasibility of generated instances, we use a flow-inspired LP that solves the feasibility problem for the instance, ignoring lower bounds ($P^-_j$).

The lower bounds might still render an instance infeasible, even if the LP indicates that a feasible flow exists. However, it provides a good and quick indication.

\begin{figure*}
	\centering
	\includegraphics{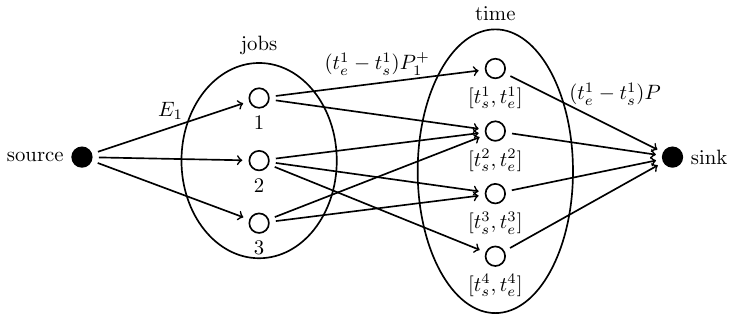}
	\caption{Example with three jobs and four intervals}
	\label{fig:bipartite-graph-example}
\end{figure*}
We will formulate the problem as a max flow problem. For this, we construct a bipartite graph, with one node for each job on one side, and one node for each time interval on the other. To obtain relevant intervals, we sort the set of all release times and deadlines and get an interval for every two consecutive distinct values. This results in at most $2n - 1$ intervals.

There is an arc between a job node ($j$) and a time interval node ($t$), if the job can be processed during the interval $t = [t_s, t_e)$ (i.e. $r_j \leq t_s$ and $t_e \leq \bar{d}_j$). The capacity of that arc is equal to $(t_e - t_s)P_j^+$, the maximum amount of power that job $j$ can receive during that interval.

We add a source that has an arc going to every job node, with $E_j$ (the total amount of work for that job) as its capacity, and a sink with an incoming arc from each time interval node with a capacity of $(t_e - t_s)P$ (the available power in that interval).

A flow through this graph now will represent a power assignment to the jobs. We use an LP analogue to solve this flow problem. If a max flow of $\sum_j E_j$ is found, we know the problem without lower bounds is feasible. If the max flow has a smaller value, we know that this instance has no feasible solution.

An example with three jobs and four time intervals is illustrated in Figure \ref{fig:bipartite-graph-example}.

\subsection{Choice of parameters for instance generation}\label{sec:param-instgen}
We performed a grid search on a number of possible combinations of parameter settings, generating 100 random instances for each setting and using the flow-inspired LP described above to check the feasibility.

This analysis was performed for different values of $n$, while $P$ was fixed at $50.00$. A parameter setting was selected for each value of $n$ such that for 99\% of generated instances a feasible flow exists in the relaxed problem (without lower bounds). The aim was to end up with mostly feasible instances, with the occasional infeasible instance mixed in.

Note, however, that this analysis was performed for $P=50.00$ only. While the same parameter settings were used for other values of $P$, the instances generated are not necessarily equally likely to be feasible. We observe in practice, however, that the likelihood of feasibility is very similar.

The parameter settings used for each value of $n$ are listed in Table \ref{tab:parameter-settings-instance-generation}.

\begin{table}
	\begin{center}
		\caption{Parameter settings used for generating instances of different sizes}
		\label{tab:parameter-settings-instance-generation}
	\begin{tabular}{@{}lrrrr@{}}
		\toprule
		 & $a_{\mathrm{maxlow}}$ & $a_{\mathrm{minupp}}$ &
		$a_{\mathrm{rshift}}$ & $a_{\mathrm{pws}}$ \\
		\midrule
		$n=5, 10$ & 0.25 & 0.25 & 0.125 & 2 \\
		$n=15, 20, 30, 50$ & 0.25 & 0.25 & 0.125 & 1.5 \\
		\botrule
	\end{tabular}
	\end{center}
\end{table}

\section{Computational results}\label{sec:results}
We ran computational experiments, measuring both solution time and the objective value of the best found solution for the MILP formulation and the hybrid local search algorithm. The instances we generated and used to run our tests can be found here: \url{https://github.com/RoelBrouwer/2022-cecsp-data/releases/tag/jos} and the code used to run the experiments is available here: \url{https://github.com/RoelBrouwer/continuousresource/releases/tag/jos}.

\subsection{Parameter settings}\label{sec:param-settings}
To run our simulated annealing algorithm, we need to determine the value for the following parameters:
\begin{itemize}
	\item Initial temperature $T_{init}$;
	\item Number of iterations between temperature updates $\alpha_{period}$;
	\item Multiplication factor for updating the temperature $\alpha$;
	\item Penalty for using slack variables $L^B, L^R$;
	\item Probability of selecting each neighborhood operator:
	\begin{itemize}
		\item Swap $p_s$;
		\item Single move $p_m$;
		\item Paired move $p_p$.
	\end{itemize}
\end{itemize}

For smaller values of $n$, we ran a (limited) grid search for a number of combinations of parameter values, selecting values that performed well in these instances, and extrapolated this to larger instance sizes. The selected values are presented in Table \ref{tab:sa-params}.

In addition, since we consider the hybrid local search to be most relevant for the instances with $n=50$, we allow restarts under certain conditions for these larger instances. If a run completes before 1800 seconds have passed, we perform 100 random swaps on the current event order, and restart the simulated annealing with the obtained order as the initial solution.

\begin{table}
	\begin{center}
	\caption{Parameter settings for simulated annealing used for computational experiments}
	\label{tab:sa-params}
	\begin{tabular}{@{}llllllll@{}}
		\toprule
		$T_{init}$ & $\alpha_{period}$ & $\alpha$ & $L^B$ & $L^R$ & $p_s$ & $p_m$ & $p_p$ \\
		\midrule
		$n$ & $(2n - 1) \cdot 4$ & 0.95 & 5 & 5 & 0.75 & 0.15 & 0.1 \\
		\botrule
	\end{tabular}
	\end{center}
\end{table}
\vspace*{1em}
For solving both the MILP formulation and the LP subproblem of the simulated annealing algorithm we used CPLEX Studio 22.1, with a time limit of 3600 seconds, and limited to the use of a single thread. Everything else has been implemented in the Python programming language. The processor of the system used to run the tests on was an Intel(R) Xeon(R) Gold 6130 CPU @ 2.10GHz.

\begin{table}
	\begin{center}
		\caption{Comparison of MILP formulation and simulated annealing algorithm on small instances}
		\label{tab:sa-mip-comparison}
		\begin{tabular}{@{}lrrrr@{}}
			\toprule
			& \multicolumn{2}{c}{MILP} & \multicolumn{2}{c}{SA} \\
			$n$ & Avg. time & Avg. diff & Avg. time & Avg. diff \\
			\midrule
			5 & 2.23 & 0.00 & 85.63 & 0.00 \\
			10 & 2300.45 & 0.00 & 280.18 & 0.00 \\
			15 & 3521.22 & 0.03 & 532.74 & 0.00 \\
			\botrule
		\end{tabular}
	\end{center}
\end{table}

\subsection{Results}
The full results of these runs can be found in Tables \ref{tab:appendix-results-5} - \ref{tab:appendix-results-50}.

In these result tables, the first four columns uniquely describe the instance: the number of jobs $n$, resource availability $P$, whether the weights are sorted in ascending order (adversarial). As four instances of each type were generated, the fourth column identifies each instance with a number 0-3. We then list whether our feasibility test was passed by the instance or not (as described in Section \ref{sec:flow}).

This is followed by a column indicating the best objective value that we have encountered in any run of any of our algorithms. Note that this include some extremely long runs done in the trail phase. Next we list the results of the MILP and hybrid local search algorithm (SA). For each we present the run time and best found objective value, while we also add the objective value of the initial solution (found by the greedy algorithm described in Section \ref{sec:init-sol}) for our hybrid local search approach.

Any cursive values indicate objective values corresponding to infeasible solutions. In the case of a solution found by our hybrid local search, this means that some slack variables have a non-zero value. An objective value of a solution found by the MILP or hybrid local search is put in boldface if it is equal to the best known objective value. Finally, we put \texttt{LIMIT} in the \textit{time} column under \textit{MILP} if the solver did not terminate within 3600 seconds, meaning it was unable to prove optimality of the current best found solution.

Note that the results for the hybrid local search in these tables are of a single run of each of the algorithms, they are not averages or best-of results.

In Table \ref{tab:sa-mip-comparison} we compare the performance of the MILP formulation and the simulated annealing on small instances ($n=5, 10,15$) in terms of run time and solution quality. This analysis is based on the results in Tables \ref{tab:appendix-results-5} - \ref{tab:appendix-results-15}. Infeasible instances were left out of the analysis. The average runtime and difference are computed only over those instances where both approaches come up with a feasible solution. Note that the average runtimes of the MILP for $n=10,15$ are somewhat misleading, as many runs have been cut off at 3600 seconds, resulting in a runtime equal to the cut-off time for these runs. The difference is defined to be the relative difference between the score of the best known solution and the reported solution. 

It is relevant to point out that for both $n=10$ and $n=15$ there are two instances where the MILP is able to find a feasible solution, and the simulated annealing is not, while for $n=15$ the simulated annealing finds a feasible solution for five instances where the MILP fails.

We can see that the runtime of the hybrid local search approach scales better than that of the MILP. Our approach is competitive with the MILP on smaller instances: it finds equally good solutions, only taking more time for the smallest instances. On larger instances, the hybrid local search finds better solutions than the MILP in less time, if the MILP is able to find any feasible solution at all. The difference is likely to become even clearer if we would perform a number of reruns of the local search and only keep the best result, as we already did for the larger instances ($n=50$). For $n=15$, for example, about 6-7 reruns can be performed without exceeding the time limit of 3600 that the MILP also has to respect.

We do see, however, that the larger instances are much more challenging. In Table \ref{tab:sa-mip-feasible}, we show the number of instances for which  we summarize the number of instances for which the (flow) feasibility test predicted a feasible solution exists, and present the number of times a feasible solution and the best known solution were found by our hybrid approach and the MILP formulation. This analysis is based on the results in Tables \ref{tab:appendix-results-5} - \ref{tab:appendix-results-50}. By allowing restarts, we are able to find feasible solutions in the majority of cases, even for $n=50$, but tackling larger instance sizes remains a challenge.

Note that the best found solution for $n=30$ is always the same as the reported solution for the hybrid local search algorithm. We added the instances of this size at a later stage to gain better insight, as the sizes of $n=20$ and $n=50$ differ strongly in their results. We did not do trail runs for these instances, and the reported solutions are therefore by default always the best known. For the other large instances ($n=20,50$), we observe that the number of times that our approach finds the best known solution strongly decreases. To put this in perspective, we note that the best known solution for these larger instances was often found during a trail run that took many hours. These do not offer a realistic alternative to the presented approach. 

The initial solution that the greedy algorithm from Section \ref{sec:init-sol} finds, provides a good starting point for our local search. For $n=15$ it occasionally even comes up with a better solution than the MILP finds in an hour. In addition to that, the local search strongly improves the initial solution for all instance sizes.

\section{Conclusion and future work}\label{sec:conclusion}
Our hybrid local search approach matches the MILP formulation in solution quality for small instances, and is able to find a feasible solution for larger instances in reasonable time. We apply a decomposition of the problem in two parts, where a local search algorithm is used to find event orders and an LP is used to find an optimal schedule for a given event order. This approach opens the door for finding solutions to larger instances. With further tuning and refinement of the current approach it is promising for finding solutions to instances of larger sizes.

One of the avenues for future work would be exactly this, further tuning the approach to deal with larger instances. The addition of more preprocessing steps (e.g. identifying more implicit precedences) and the experimentation with different search strategies (e.g. adding some form of random restart) are promising directions for further investigation. The MILP formulation can be further strengthened by adding valid inequalities.

Finally, one could look at further extensions of the problem. Among those of interest are the addition of precedence relations and changing $P$ from a constant into a variable resource availability function $P(t)$.

\begin{table}
	\begin{center}
		\caption{Number of feasible solutions found for each instance size}
		\label{tab:sa-mip-feasible}
		\begin{tabular}{@{}lrrrrr@{}}
			\toprule
			& Flow & \multicolumn{2}{c}{SA} & \multicolumn{2}{c}{MIP} \\
			$n$ & Feas. & Feas. & Best & Feas. & Best \\
			\midrule
			5 & 28 & 28 & 28 & 28 & 28 \\
			10 & 32 & 30 & 26 & 32 & 23 \\
			15 & 32 & 30 & 15 & 27 & 4 \\
			20 & 31 & 30 & 8 & - & - \\
			30 & 32 & 28 & 32 & - & - \\
			50 & 32 & 21 & 9 & - & - \\
			\botrule
		\end{tabular}
	\end{center}
\end{table}

\begin{table*}
	\begin{center}
		\caption{Computational results for $n=5$}
		\label{tab:appendix-results-5}
		\begin{tabular}{@{}llll|l|l|rr|rrr@{}}
			\toprule
			\multicolumn{4}{c|}{} & \multicolumn{1}{c|}{Flow}& \multicolumn{1}{c|}{Best}& \multicolumn{2}{c|}{MILP}& \multicolumn{3}{c}{SA} \\
			$n$ & $P$ & adv.? & \# & Feas.? & known & time & obj & time & obj & init \\
			\midrule
			5 & 25 & 0 & 0 & yes & 163.58 & 2.04 & \textbf{163.58} & 24.18 & \textbf{163.58} & 177.77\\
			5 & 25 & 0 & 1 & yes & 165.72 & 1.58 & \textbf{165.72} & 189.19 & \textbf{165.72} & 212.33\\
			5 & 25 & 0 & 2 & yes & 99.42 & 18.87 & \textbf{99.42} & 59.68 & \textbf{99.42} & 104.79\\
			5 & 25 & 0 & 3 & yes & 78.70 & 1.92 & \textbf{78.70} & 80.97 & \textbf{78.70} & 97.96\\
			5 & 25 & 1 & 0 & yes & 113.21 & 1.96 & \textbf{113.21} & 104.72 & \textbf{113.21} & 132.90\\
			5 & 25 & 1 & 1 & yes & 61.94 & 6.30 & \textbf{61.94} & 91.03 & \textbf{61.94} & 68.71\\
			5 & 25 & 1 & 2 & yes & 73.06 & 1.83 & \textbf{73.06} & 111.09 & \textbf{73.06} & 88.02\\
			5 & 25 & 1 & 3 & yes & 96.81 & 6.36 & \textbf{96.81} & 95.73 & \textbf{96.81} & 98.34\\
			5 & 50 & 0 & 0 & yes & 72.39 & 1.56 & \textbf{72.39} & 75.05 & \textbf{72.39} & 76.19\\
			5 & 50 & 0 & 1 & yes & 88.61 & 0.87 & \textbf{88.61} & 103.08 & \textbf{88.61} & 94.87\\
			5 & 50 & 0 & 2 & yes & 95.38 & 1.13 & \textbf{95.38} & 100.95 & \textbf{95.38} & 101.10\\
			5 & 50 & 0 & 3 & yes & 80.81 & 1.38 & \textbf{80.81} & 130.12 & \textbf{80.81} & 92.27\\
			5 & 50 & 1 & 0 & yes & 98.25 & 4.23 & \textbf{98.25} & 121.34 & \textbf{98.25} & 112.77\\
			5 & 50 & 1 & 1 & yes & 74.93 & 2.24 & \textbf{74.93} & 101.55 & \textbf{74.93} & 82.42\\
			5 & 50 & 1 & 2 & yes & 83.88 & 0.94 & \textbf{83.88} & 92.08 & \textbf{83.88} & 99.66\\
			5 & 50 & 1 & 3 & yes & 102.10 & 2.08 & \textbf{102.10} & 91.60 & \textbf{102.10} & 102.95\\
			5 & 100 & 0 & 0 & yes & 75.25 & 0.66 & \textbf{75.25} & 91.27 & \textbf{75.25} & 91.20\\
			5 & 100 & 0 & 1 & yes & 77.71 & 0.51 & \textbf{77.71} & 36.22 & \textbf{77.71} & 175.64\\
			5 & 100 & 0 & 2 & yes & 49.32 & 0.95 & \textbf{49.32} & 62.31 & \textbf{49.32} & 56.07\\
			5 & 100 & 0 & 3 & yes & 51.97 & 0.54 & \textbf{51.97} & 64.66 & \textbf{51.97} & 111.72\\
			5 & 100 & 1 & 0 & yes & 53.80 & 0.45 & \textbf{53.80} & 123.77 & \textbf{53.80} & 54.48\\
			5 & 100 & 1 & 1 & yes & 69.92 & 0.71 & \textbf{69.92} & 43.58 & \textbf{69.92} & 76.33\\
			5 & 100 & 1 & 2 & yes & 93.13 & 0.92 & \textbf{93.13} & 42.37 & \textbf{93.13} & 100.24\\
			5 & 100 & 1 & 3 & yes & 53.79 & 0.64 & \textbf{53.79} & 39.44 & \textbf{53.79} & 55.46\\
			5 & 200 & 0 & 0 & yes & 67.13 & 0.38 & \textbf{67.13} & 89.15 & \textbf{67.13} & 137.92\\
			5 & 200 & 0 & 1 & \textit{no} & \textit{120.94} & 0.43 & \textit{-1.00} & 0.24 & \textit{169.91} & 169.91\\
			5 & 200 & 0 & 2 & \textit{no} & \textit{74.41} & 0.45 & \textit{-1.00} & 1.78 & \textit{74.44} & 146.13\\
			5 & 200 & 0 & 3 & yes & 57.02 & 0.48 & \textbf{57.02} & 86.05 & \textbf{57.02} & 59.86\\
			5 & 200 & 1 & 0 & yes & 56.35 & 0.39 & \textbf{56.35} & 67.70 & \textbf{56.35} & 57.73\\
			5 & 200 & 1 & 1 & \textit{no} & \textit{81.35} & 0.45 & \textit{-1.00} & 1.62 & \textit{81.61} & 148.08\\
			5 & 200 & 1 & 2 & yes & 67.19 & 0.61 & \textbf{67.19} & 78.81 & \textbf{67.19} & 122.79\\
			5 & 200 & 1 & 3 & \textit{no} & \textit{229.58} & 0.41 & \textit{-1.00} & 1.16 & \textit{231.03} & 330.74\\
			\botrule
		\end{tabular}
	\end{center}
\end{table*}

\begin{table*}
	\begin{center}
		\caption{Computational results for $n=10$}
		\label{tab:appendix-results-10}
		\begin{tabular}{@{}llll|l|l|rr|rrr@{}}
			\toprule
			\multicolumn{4}{c|}{} & \multicolumn{1}{c|}{Flow}& \multicolumn{1}{c|}{Best}& \multicolumn{2}{c|}{MILP}& \multicolumn{3}{c}{SA} \\
			$n$ & $P$ & adv.? & \# & Feas.? & known & time & obj & time & obj & init \\
			\midrule
			10 & 25 & 0 & 0 & yes & 359.47 & \texttt{LIMIT} & 360.49 & 193.25 & \textbf{359.47} & 424.51\\
			10 & 25 & 0 & 1 & yes & \textit{334.63} & \texttt{LIMIT} & 337.72 & 206.95 & \textit{334.63} & 424.35\\
			10 & 25 & 0 & 2 & yes & 225.99 & \texttt{LIMIT} & 226.76 & 204.47 & 232.05 & 250.28\\
			10 & 25 & 0 & 3 & yes & 300.08 & \texttt{LIMIT} & \textbf{300.08} & 243.97 & \textbf{300.08} & 392.84\\
			10 & 25 & 1 & 0 & yes & 345.71 & 525.62 & \textbf{345.71} & 277.95 & \textbf{345.71} & 401.68\\
			10 & 25 & 1 & 1 & yes & 394.69 & 3228.46 & \textbf{394.69} & 218.89 & \textbf{394.69} & 446.65\\
			10 & 25 & 1 & 2 & yes & 399.06 & \texttt{LIMIT} & \textbf{399.06} & 245.50 & \textbf{399.06} & 493.58\\
			10 & 25 & 1 & 3 & yes & 378.15 & \texttt{LIMIT} & \textbf{378.15} & 280.72 & \textbf{378.15} & 400.00\\
			10 & 50 & 0 & 0 & yes & 193.57 & \texttt{LIMIT} & \textbf{193.57} & 293.98 & \textbf{193.57} & 207.17\\
			10 & 50 & 0 & 1 & yes & 220.46 & \texttt{LIMIT} & 221.18 & 283.44 & 220.87 & 239.74\\
			10 & 50 & 0 & 2 & yes & 254.38 & 1131.58 & \textbf{254.38} & 332.78 & \textbf{254.38} & 289.27\\
			10 & 50 & 0 & 3 & yes & 191.05 & \texttt{LIMIT} & \textbf{191.05} & 248.19 & \textbf{191.05} & 213.67\\
			10 & 50 & 1 & 0 & yes & 162.90 & \texttt{LIMIT} & \textbf{162.90} & 281.60 & \textbf{162.90} & 179.64\\
			10 & 50 & 1 & 1 & yes & 171.90 & \texttt{LIMIT} & 172.04 & 298.19 & \textbf{171.90} & 174.68\\
			10 & 50 & 1 & 2 & yes & 194.43 & \texttt{LIMIT} & 194.56 & 1.17 & \textit{552.53} & 595.79\\
			10 & 50 & 1 & 3 & yes & 289.80 & \texttt{LIMIT} & \textbf{289.80} & 256.29 & 303.90 & 388.97\\
			10 & 100 & 0 & 0 & yes & 155.15 & 29.48 & \textbf{155.15} & 312.78 & \textbf{155.15} & 162.38\\
			10 & 100 & 0 & 1 & yes & 139.62 & \texttt{LIMIT} & \textbf{139.62} & 308.31 & \textbf{139.62} & 148.31\\
			10 & 100 & 0 & 2 & yes & 203.37 & \texttt{LIMIT} & \textbf{203.37} & 284.11 & \textbf{203.37} & 211.90\\
			10 & 100 & 0 & 3 & yes & 168.99 & \texttt{LIMIT} & 171.55 & 306.39 & \textbf{168.99} & 194.37\\
			10 & 100 & 1 & 0 & yes & 205.44 & \texttt{LIMIT} & \textbf{205.44} & 348.42 & \textbf{205.44} & 219.86\\
			10 & 100 & 1 & 1 & yes & 171.26 & \texttt{LIMIT} & 171.42 & 393.40 & \textbf{171.26} & 225.93\\
			10 & 100 & 1 & 2 & yes & 143.83 & \texttt{LIMIT} & \textbf{143.83} & 253.98 & \textbf{143.83} & 163.53\\
			10 & 100 & 1 & 3 & yes & 145.66 & \texttt{LIMIT} & 148.32 & 282.35 & 148.32 & 162.14\\
			10 & 200 & 0 & 0 & yes & 125.38 & 181.10 & \textbf{125.38} & 221.56 & \textbf{125.38} & 130.17\\
			10 & 200 & 0 & 1 & yes & 138.32 & 4.58 & \textbf{138.32} & 309.86 & \textbf{138.32} & 185.09\\
			10 & 200 & 0 & 2 & yes & 107.74 & 33.39 & \textbf{107.74} & 276.53 & \textbf{107.74} & 113.15\\
			10 & 200 & 0 & 3 & yes & 100.16 & 4.42 & \textbf{100.16} & 314.35 & \textbf{100.16} & 104.39\\
			10 & 200 & 1 & 0 & yes & 158.68 & 18.55 & \textbf{158.68} & 307.19 & \textbf{158.68} & 182.99\\
			10 & 200 & 1 & 1 & yes & 135.62 & 13.38 & \textbf{135.62} & 443.79 & \textbf{135.62} & 136.47\\
			10 & 200 & 1 & 2 & yes & 108.72 & 7.00 & \textbf{108.72} & 193.36 & \textbf{108.72} & 109.42\\
			10 & 200 & 1 & 3 & yes & 101.19 & \texttt{LIMIT} & \textbf{101.19} & 233.21 & \textbf{101.19} & 106.54\\
			\botrule
		\end{tabular}
	\end{center}
\end{table*}

\begin{table*}
	\begin{center}
		\caption{Computational results for $n=15$}
		\label{tab:appendix-results-15}
		\begin{tabular}{@{}llll|l|l|rr|rrr@{}}
			\toprule
			\multicolumn{4}{c|}{} & \multicolumn{1}{c|}{Flow}& \multicolumn{1}{c|}{Best}& \multicolumn{2}{c|}{MILP}& \multicolumn{3}{c}{SA} \\
			$n$ & $P$ & adv.? & \# & Feas.? & known & time & obj & time & obj & init \\
			\midrule
			15 & 25 & 0 & 0 & yes & 831.86 & \texttt{LIMIT} & 897.55 & 538.08 & \textbf{831.86} & 871.10\\
			15 & 25 & 0 & 1 & yes & 587.90 & \texttt{LIMIT} & \textit{-1.00} & 390.93 & 587.96 & 752.17\\
			15 & 25 & 0 & 2 & yes & 779.20 & \texttt{LIMIT} & 790.29 & 399.92 & \textbf{779.20} & 985.93\\
			15 & 25 & 0 & 3 & yes & 800.74 & \texttt{LIMIT} & 825.81 & 417.82 & \textbf{800.74} & 997.23\\
			15 & 25 & 1 & 0 & yes & 733.29 & \texttt{LIMIT} & 770.22 & 326.79 & \textbf{733.29} & 882.03\\
			15 & 25 & 1 & 1 & yes & 805.43 & \texttt{LIMIT} & 898.95 & 536.96 & 811.87 & 953.44\\
			15 & 25 & 1 & 2 & yes & 649.83 & \texttt{LIMIT} & \textit{-1.00} & 563.08 & \textbf{649.83} & 847.58\\
			15 & 25 & 1 & 3 & yes & 561.05 & \texttt{LIMIT} & 572.22 & 402.05 & 561.98 & 662.60\\
			15 & 50 & 0 & 0 & yes & 542.53 & \texttt{LIMIT} & \textit{-1.00} & 448.71 & 542.57 & 634.80\\
			15 & 50 & 0 & 1 & yes & 309.19 & \texttt{LIMIT} & 311.46 & 544.40 & \textbf{309.19} & 363.38\\
			15 & 50 & 0 & 2 & yes & 353.07 & \texttt{LIMIT} & 358.26 & 411.76 & 353.08 & 382.82\\
			15 & 50 & 0 & 3 & yes & 391.34 & \texttt{LIMIT} & \textit{-1.00} & 462.41 & 391.48 & 611.09\\
			15 & 50 & 1 & 0 & yes & 322.26 & \texttt{LIMIT} & 325.22 & 520.40 & \textbf{322.26} & 349.78\\
			15 & 50 & 1 & 1 & yes & 475.77 & \texttt{LIMIT} & 492.32 & 2.18 & \textit{793.14} & 793.14\\
			15 & 50 & 1 & 2 & yes & 589.86 & \texttt{LIMIT} & \textit{-1.00} & 436.84 & \textbf{589.86} & 626.66\\
			15 & 50 & 1 & 3 & yes & 347.26 & \texttt{LIMIT} & 349.15 & 392.01 & \textbf{347.26} & 399.45\\
			15 & 100 & 0 & 0 & yes & 235.75 & \texttt{LIMIT} & 246.03 & 427.83 & 239.26 & 250.06\\
			15 & 100 & 0 & 1 & yes & 330.64 & \texttt{LIMIT} & 337.54 & 507.91 & 331.21 & 338.12\\
			15 & 100 & 0 & 2 & yes & 314.33 & \texttt{LIMIT} & 321.52 & 643.93 & \textbf{314.33} & 348.71\\
			15 & 100 & 0 & 3 & yes & 288.97 & \texttt{LIMIT} & 289.10 & 479.21 & \textbf{288.97} & 342.44\\
			15 & 100 & 1 & 0 & yes & 279.48 & \texttt{LIMIT} & 291.05 & 542.24 & 279.58 & 299.08\\
			15 & 100 & 1 & 1 & yes & 341.46 & \texttt{LIMIT} & 345.56 & 671.86 & \textbf{341.46} & 408.68\\
			15 & 100 & 1 & 2 & yes & 215.79 & \texttt{LIMIT} & 230.71 & 2.30 & \textit{520.72} & 520.72\\
			15 & 100 & 1 & 3 & yes & 259.86 & \texttt{LIMIT} & 265.08 & 546.11 & \textbf{259.86} & 296.06\\
			15 & 200 & 0 & 0 & yes & 168.23 & \texttt{LIMIT} & 190.76 & 579.06 & 168.24 & 181.16\\
			15 & 200 & 0 & 1 & yes & 253.88 & \texttt{LIMIT} & 261.80 & 881.63 & \textbf{253.88} & 259.09\\
			15 & 200 & 0 & 2 & yes & 143.38 & \texttt{LIMIT} & 144.71 & 702.94 & 143.51 & 151.02\\
			15 & 200 & 0 & 3 & yes & 237.71 & 1496.19 & \textbf{237.71} & 549.86 & \textbf{237.71} & 243.81\\
			15 & 200 & 1 & 0 & yes & 179.79 & \texttt{LIMIT} & 179.86 & 646.97 & 179.85 & 199.15\\
			15 & 200 & 1 & 1 & yes & 198.08 & \texttt{LIMIT} & \textbf{198.08} & 526.01 & 198.10 & 207.30\\
			15 & 200 & 1 & 2 & yes & 160.21 & \texttt{LIMIT} & \textbf{160.21} & 565.79 & 160.34 & 165.41\\
			15 & 200 & 1 & 3 & yes & 244.38 & \texttt{LIMIT} & \textbf{244.38} & 556.86 & 244.41 & 254.76\\
			\botrule
		\end{tabular}
	\end{center}
\end{table*}

\begin{table*}
	\begin{center}
		\caption{Computational results for $n=20$}
		\label{tab:appendix-results-20}
		\begin{tabular}{@{}llll|l|l|rrr@{}}
			\toprule
			\multicolumn{4}{c|}{} & \multicolumn{1}{c|}{Flow}& \multicolumn{1}{c|}{Best}& \multicolumn{3}{c}{SA} \\
			$n$ & $P$ & adv.? & \# & Feas.? & known & time & obj & init \\
			\midrule
			20 & 25 & 0 & 0 & yes & 1122.08 & 862.19 & 1125.55 & 1643.66\\
			20 & 25 & 0 & 1 & yes & 1343.89 & 685.22 & \textbf{1343.89} & 1771.98\\
			20 & 25 & 0 & 2 & yes & 1012.18 & 877.71 & \textbf{1012.18} & 1197.75\\
			20 & 25 & 0 & 3 & \textit{no} & \textit{1445.11} & 732.78 & \textit{1518.36} & 1784.42\\
			20 & 25 & 1 & 0 & yes & 991.01 & 776.51 & \textit{995.64} & 1195.65\\
			20 & 25 & 1 & 1 & yes & 1316.27 & 797.72 & \textbf{1316.27} & 1768.37\\
			20 & 25 & 1 & 2 & yes & 1094.52 & 549.41 & \textbf{1094.52} & 1399.41\\
			20 & 25 & 1 & 3 & yes & 991.20 & 702.07 & \textbf{991.20} & 1123.76\\
			20 & 50 & 0 & 0 & yes & 657.76 & 752.77 & \textbf{657.76} & 827.27\\
			20 & 50 & 0 & 1 & yes & 765.34 & 712.16 & 766.59 & 951.32\\
			20 & 50 & 0 & 2 & yes & 590.19 & 809.64 & \textbf{590.19} & 620.57\\
			20 & 50 & 0 & 3 & yes & 729.51 & 710.34 & \textbf{729.51} & 1049.81\\
			20 & 50 & 1 & 0 & yes & 571.33 & 678.59 & 571.48 & 635.55\\
			20 & 50 & 1 & 1 & yes & 625.52 & 674.42 & 628.49 & 755.16\\
			20 & 50 & 1 & 2 & yes & 548.63 & 565.51 & 548.74 & 737.98\\
			20 & 50 & 1 & 3 & yes & 551.01 & 603.31 & 554.73 & 610.62\\
			20 & 100 & 0 & 0 & yes & 301.95 & 989.00 & 302.12 & 314.61\\
			20 & 100 & 0 & 1 & yes & 476.50 & 810.92 & 476.71 & 527.38\\
			20 & 100 & 0 & 2 & yes & 363.01 & 902.16 & 363.12 & 443.13\\
			20 & 100 & 0 & 3 & yes & 488.63 & 1461.01 & 488.68 & 506.45\\
			20 & 100 & 1 & 0 & yes & 365.24 & 980.08 & 365.31 & 501.94\\
			20 & 100 & 1 & 1 & yes & 493.23 & 1072.18 & 493.28 & 584.07\\
			20 & 100 & 1 & 2 & yes & 415.11 & 1185.18 & 415.24 & 454.11\\
			20 & 100 & 1 & 3 & yes & 359.19 & 664.77 & 359.25 & 482.25\\
			20 & 200 & 0 & 0 & yes & 303.41 & 818.06 & 303.78 & 309.97\\
			20 & 200 & 0 & 1 & yes & 310.99 & 883.86 & 311.31 & 364.42\\
			20 & 200 & 0 & 2 & yes & 345.24 & 1064.08 & 345.28 & 367.84\\
			20 & 200 & 0 & 3 & yes & 285.33 & 796.83 & 285.70 & 290.13\\
			20 & 200 & 1 & 0 & yes & 305.00 & 1095.02 & 305.06 & 334.49\\
			20 & 200 & 1 & 1 & yes & 337.89 & 753.72 & 337.97 & 346.15\\
			20 & 200 & 1 & 2 & yes & 313.49 & 828.79 & 313.72 & 332.67\\
			20 & 200 & 1 & 3 & yes & 318.78 & 1098.68 & 318.96 & 365.87\\
			\botrule
		\end{tabular}
	\end{center}
\end{table*}

\begin{table*}
	\begin{center}
		\caption{Computational results for $n=30$}
		\label{tab:appendix-results-30}
		\begin{tabular}{@{}llll|l|l|rrr@{}}
			\toprule
			\multicolumn{4}{c|}{} & \multicolumn{1}{c|}{Flow}& \multicolumn{1}{c|}{Best}& \multicolumn{3}{c}{SA} \\
			$n$ & $P$ & adv.? & \# & Feas.? & known & time & obj & init \\
			\midrule
			30 & 25 & 0 & 0 & yes & 2088.94 & 1502.01 & \textbf{2088.94} & 2599.09\\
			30 & 25 & 0 & 1 & yes & \textit{3134.71} & 5.58 & \textit{\textbf{3134.71}} & 3211.39\\
			30 & 25 & 0 & 2 & yes & 2291.27 & 1406.69 & \textbf{2291.27} & 2744.97\\
			30 & 25 & 0 & 3 & yes & 2082.82 & 1424.99 & \textbf{2082.82} & 2860.31\\
			30 & 25 & 1 & 0 & yes & 2920.08 & 1142.54 & \textbf{2920.08} & 3380.82\\
			30 & 25 & 1 & 1 & yes & 2963.77 & 1483.27 & \textbf{2963.77} & 3870.67\\
			30 & 25 & 1 & 2 & yes & 3436.78 & 29.96 & \textbf{3436.78} & 3622.45\\
			30 & 25 & 1 & 3 & yes & 3237.29 & 1760.39 & \textbf{3237.29} & 4677.28\\
			30 & 50 & 0 & 0 & yes & \textit{2426.46} & 5.62 & \textit{\textbf{2426.46}} & 2476.53\\
			30 & 50 & 0 & 1 & yes & \textit{2538.60} & 5.75 & \textit{\textbf{2538.60}} & 2589.17\\
			30 & 50 & 0 & 2 & yes & 1052.26 & 1728.24 & \textbf{1052.26} & 1254.98\\
			30 & 50 & 0 & 3 & yes & \textit{2019.35} & 5.69 & \textit{\textbf{2019.35}} & 2020.03\\
			30 & 50 & 1 & 0 & yes & 1660.95 & 1601.76 & \textbf{1660.95} & 2041.24\\
			30 & 50 & 1 & 1 & yes & 1732.74 & 1614.14 & \textbf{1732.74} & 2391.34\\
			30 & 50 & 1 & 2 & yes & 1778.19 & 1381.26 & \textbf{1778.19} & 2190.11\\
			30 & 50 & 1 & 3 & yes & 1703.09 & 308.90 & \textbf{1703.09} & 2080.97\\
			30 & 100 & 0 & 0 & yes & 747.01 & 1591.51 & \textbf{747.01} & 1482.49\\
			30 & 100 & 0 & 1 & yes & 904.57 & 1628.79 & \textbf{904.57} & 1262.40\\
			30 & 100 & 0 & 2 & yes & 779.62 & 2140.44 & \textbf{779.62} & 929.26\\
			30 & 100 & 0 & 3 & yes & 865.12 & 1481.99 & \textbf{865.12} & 1160.84\\
			30 & 100 & 1 & 0 & yes & 836.81 & 2110.24 & \textbf{836.81} & 1118.03\\
			30 & 100 & 1 & 1 & yes & 976.40 & 1890.41 & \textbf{976.40} & 1419.84\\
			30 & 100 & 1 & 2 & yes & 754.56 & 1861.07 & \textbf{754.56} & 1085.14\\
			30 & 100 & 1 & 3 & yes & 695.52 & 1690.72 & \textbf{695.52} & 902.68\\
			30 & 200 & 0 & 0 & yes & 497.55 & 2059.61 & \textbf{497.55} & 549.87\\
			30 & 200 & 0 & 1 & yes & 486.30 & 1826.58 & \textbf{486.30} & 513.94\\
			30 & 200 & 0 & 2 & yes & 509.32 & 1685.78 & \textbf{509.32} & 552.84\\
			30 & 200 & 0 & 3 & yes & 538.68 & 2378.16 & \textbf{538.68} & 604.61\\
			30 & 200 & 1 & 0 & yes & 468.25 & 2631.23 & \textbf{468.25} & 623.84\\
			30 & 200 & 1 & 1 & yes & 586.33 & 1823.30 & \textbf{586.33} & 715.84\\
			30 & 200 & 1 & 2 & yes & 619.83 & 1826.30 & \textbf{619.83} & 645.95\\
			30 & 200 & 1 & 3 & yes & 568.19 & 2356.85 & \textbf{568.19} & 687.19\\
			\botrule
		\end{tabular}
	\end{center}
\end{table*}

\begin{table*}
	\begin{center}
		\caption{Computational results for $n=50$}
		\label{tab:appendix-results-50}
		\begin{tabular}{@{}llll|l|l|rrr@{}}
			\toprule
			\multicolumn{4}{c|}{} & \multicolumn{1}{c|}{Flow}& \multicolumn{1}{c|}{Best}& \multicolumn{3}{c}{SA} \\
			$n$ & $P$ & adv.? & \# & Feas.? & known & time & obj & init \\
			\midrule
			50 & 25 & 0 & 0 & yes & 6002.37 & 4293.20 & 6777.69 & 7958.75\\
			50 & 25 & 0 & 1 & yes & 6573.68 & 4665.56 & 7245.46 & 7509.22\\
			50 & 25 & 0 & 2 & yes & \textit{6423.54} & 5003.16 & \textit{\textbf{6423.54}} & 7689.15\\
			50 & 25 & 0 & 3 & yes & 7027.25 & 5916.94 & \textbf{7027.25} & 8145.16\\
			50 & 25 & 1 & 0 & yes & 6562.65 & 3602.98 & \textit{7196.96} & 7217.22\\
			50 & 25 & 1 & 1 & yes & 6257.97 & 5594.27 & \textit{7278.21} & 7520.76\\
			50 & 25 & 1 & 2 & yes & \textit{7239.58} & 6254.32 & \textit{\textbf{7239.58}} & 8007.63\\
			50 & 25 & 1 & 3 & yes & 7911.54 & 5082.28 & 7925.13 & 8873.06\\
			50 & 50 & 0 & 0 & yes & \textit{4176.25} & 5303.19 & 4469.83 & 5420.81\\
			50 & 50 & 0 & 1 & yes & \textit{5837.52} & 3731.27 & \textit{5857.73} & 5871.31\\
			50 & 50 & 0 & 2 & yes & 3167.27 & 2318.80 & \textit{3756.45} & 4459.79\\
			50 & 50 & 0 & 3 & yes & 3429.76 & 3484.34 & 4080.28 & 4731.48\\
			50 & 50 & 1 & 0 & yes & 3503.36 & 6897.78 & 3516.58 & 4041.55\\
			50 & 50 & 1 & 1 & yes & 3081.82 & 6070.92 & \textbf{3081.82} & 4950.52\\
			50 & 50 & 1 & 2 & yes & \textit{4326.41} & 2764.69 & \textit{4499.39} & 4509.67\\
			50 & 50 & 1 & 3 & yes & 3548.99 & 3541.65 & \textit{4359.87} & 4750.75\\
			50 & 100 & 0 & 0 & yes & 1609.52 & 8184.16 & \textbf{1609.52} & 1928.63\\
			50 & 100 & 0 & 1 & yes & 2205.73 & 8174.82 & \textbf{2205.73} & 2588.95\\
			50 & 100 & 0 & 2 & yes & 1658.84 & 3838.16 & \textit{2074.60} & 2127.92\\
			50 & 100 & 0 & 3 & yes & 1947.59 & 2484.76 & 2282.36 & 2622.42\\
			50 & 100 & 1 & 0 & yes & 1628.10 & 4110.64 & 1933.85 & 2240.80\\
			50 & 100 & 1 & 1 & yes & 1716.27 & 9738.11 & \textbf{1716.27} & 2096.80\\
			50 & 100 & 1 & 2 & yes & 1637.18 & 6344.71 & \textbf{1637.18} & 2402.85\\
			50 & 100 & 1 & 3 & yes & 1880.46 & 6946.31 & \textbf{1880.46} & 2358.75\\
			50 & 200 & 0 & 0 & yes & 1139.72 & 2123.86 & \textit{1384.30} & 1415.43\\
			50 & 200 & 0 & 1 & yes & 1237.01 & 2861.69 & 1352.59 & 1352.59\\
			50 & 200 & 0 & 2 & yes & 1298.67 & 12194.00 & 1306.81 & 1439.92\\
			50 & 200 & 0 & 3 & yes & 1211.63 & 9607.54 & 1219.67 & 1442.92\\
			50 & 200 & 1 & 0 & yes & 1052.00 & 13363.99 & 1060.47 & 1169.24\\
			50 & 200 & 1 & 1 & yes & 1265.67 & 2003.09 & \textit{1458.85} & 1459.74\\
			50 & 200 & 1 & 2 & yes & 1060.45 & 9351.82 & 1070.15 & 1223.95\\
			50 & 200 & 1 & 3 & yes & 1095.45 & 8013.34 & 1103.34 & 1216.94\\
			\botrule
		\end{tabular}
	\end{center}
\end{table*}





\bibliography{references}


\end{document}